\documentclass[12pt]{amsart}
\usepackage{hyperref}

\newtheorem{thm}{Theorem}

\newtheorem{prop}[thm]{Proposition}
\newtheorem{lem}[thm]{Lemma}

\theoremstyle{remark}

\numberwithin{equation}{section}

\newcommand{\R}{\mathbb{R}}
\newcommand{\flag}{\mathrm{Flag(\mathbb{R}^n)}}
\newcommand{\psl}{\mathrm{PSL}_n(\mathbb{R})}
\newcommand{\SL}{\mathrm{SL}_n(\mathbb{R})}
\newcommand{\twis}{ \mathcal{Z}_{\mathrm{tw}}(\lambda;{\mathbb{R}}^{n-1})}

\newcommand{\Hit}{\mathrm{Hit}}

\begin{document}
\title[Shearing Deformations and the symplectic structure]{Shearing deformations of Hitchin representations  and the Atiyah-Bott-Goldman symplectic form}

\author[H. Zeybek]{Hat\.{\i}ce Zeybek}
\address{Hacettepe University, Department of Mathematics,
06800 Ankara, Turkey}
\thanks{This research was supported by the Scientific and Technological
Research Council of Turkey, T\"{U}B\.{I}TAK (B\.{I}DEB-2214/A)}
 \email{haticezeybek@hacettepe.edu.tr}

\begin{abstract}
The Hitchin component  $\Hit_n(S)$ of a closed surface $S$ is a preferred component of the character variety $\mathcal{X}_{\mathrm{PSL}_n(\mathbb{R})}(S)$ consisting of homomorphisms from the fundamental group $\pi_1(S)$ to the Lie group $\psl$, whose elements enjoy remarkable geometric and dynamical properties. We consider a certain type of deformations of the elements of $\Hit_n(S)$, called shearing deformations, and compute their pairing for the Atiyah-Bott-Goldman symplectic form of the character variety.
\end{abstract}

\maketitle

For a closed, connected, oriented surface $S$ of genus at least 2, the Hitchin component $\Hit_n(S)$ is a preferred component of the character variety 
\[\mathcal{X}_{\mathrm{PSL}_n(\mathbb{R})}(S)=\{\mathrm{homomorphisms} \,\, \rho : \pi_1(S)\rightarrow \mathrm{PSL}_n(\mathbb{R})\} /\kern-3pt/ \mathrm{PSL}_n(\mathbb{R}) \]
 consisting of all group homomorphisms $\pi_1(S)\rightarrow \mathrm{PSL}_n(\mathbb{R})$ from the fundamental group $\pi_1(S)$ to the Lie group $\mathrm{PSL}_n(\mathbb{R})$, considered up to conjugation by elements of  $\mathrm{PSL}_n(\mathbb{R})$; see \S \ref{sect: Hitchin} for a precise definition. When $n=2$, the Hitchin component $\Hit_2(S)$ is just the {Teichm\"uller component} of $\mathcal{X}_{\mathrm{PSL}_2(\mathbb{R})}(S)$, consisting of the monodromies of hyperbolic metrics on the oriented surface $S$, and consequently plays an important role in 2-dimensional hyperbolic geometry and in complex analysis.  In the general case the Hitchin characters, namely the elements of the Hitchin component $\Hit_n(S)$, similarly enjoy many powerful geometric properties  from a variety of viewpoints, involving differential geometry, dynamical systems or the theory of Lie groups; see for instance the fundamental results of  Hitchin \cite{Hi}, Labourie \cite{Lab} and Fock-Goncharov \cite{FoG}. In particular, using the point of view of Higgs bundles, Hitchin proved in \cite{Hi} that $\Hit_n(S)$ is diffeomorphic to the Euclidean space $\R^{2(g-1)(n^2-1)}$.
 
 In this article, we consider the symplectic properties of a certain type of deformations of Hitchin characters called shearing deformations.  More precisely, the  character variety  $\mathcal{X}_{\mathrm{PSL}_n(\mathbb{R})}(S)$ comes with a natural symplectic form 
\[\omega : T_\rho\mathcal{X}_{\mathrm{PSL}_n(\mathbb{R})}(S) \times T_\rho\mathcal{X}_{\mathrm{PSL}_n(\mathbb{R})}(S)\rightarrow \mathbb{R},\]
the Atiyah-Bott-Goldman symplectic form \cite{AB, Go}; see also \S \ref{sect: ABG}. The main purpose of the article is to compute the pairing of two infinitesimal shearing deformations under this symplectic form.

Shearing deformations are the natural higher dimensional generalization of Thurston's earthquakes \cite{Th2} and shearing deformations \cite{Th3, Bon2} for the Teichm\"uller space $\Hit_2(S)$. A shearing deformation of a Hitchin character $\rho \in \Hit_n(S)$ is determined by two pieces of data: a geodesic lamination $\lambda$; and a twisted $\R^{n-1}$-valued transverse cocycle $\alpha \in \twis$ assigning an element of $\R^{n-1}$ to each oriented arc transverse to $\lambda$. See \cite{Dr} or \cite{BonDr2} for a precise construction, and  \S\ref{sect:ShearingDef} for the definitions that will be needed here. 

Shearing deformations also occur in the Bonahon-Dreyer parametrization \cite{BonDr1,  BonDr2}  of $\Hit_n(S)$ associated to a maximal geodesic lamination $\lambda$. This parametrization is based on the consideration of two types of invariants for a Hitchin character $\rho \in \Hit_n(S)$. The first invariants are the so-called triangle invariants $\tau_{abc}^\rho(s)$ associated to the spikes $s$ of the complement $S- \lambda$. The remaining invariants  are provided by a certain shearing relative cycle $\sigma_\rho \in \mathcal Z_{\mathrm{tw}}(\lambda, \mathrm{slits};  \R^{n-1})$, valued in an affine space whose corresponding vector space is the space  $\twis$ of twisted $\R^{n-1}$-valued transverse cocycles. In this framework, the shearing deformations of $\rho \in \Hit_n(S)$ along a maximal geodesic lamination $\lambda$ are precisely the deformations that leave the triangle invariants constant. We are therefore computing the restriction of the Atiyah-Bott-Goldman symplectic form $\omega$ to the leaves of a certain foliation of $\Hit_n(S)$. The leaves of this foliation have dimension $6(g-1)(n-1)+\lfloor \frac{1}{2}(n-1)\rfloor$, where $\lfloor x \rfloor$ denotes the largest integer that is less than or equal to $x$; this should be compared with the dimension $2(g-1)(n^2-1)$ of $\Hit_n(S)$.

Our computation can be expressed in two different ways, one that is more conceptual and one that is more practical. The best way to understand a twisted $\R^{n-1}$-valued transverse cocycle $\alpha \in \twis$ is to consider a train track $\Phi$ carrying the geodesic lamination $\lambda$. The geodesic lamination $\lambda$ has a unique orientation cover $\widehat \lambda \to \lambda$ where the leaves are oriented, and this cover uniquely extends to a two-fold cover $\widehat \Phi \to \Phi$. Then a twisted $\R^{n-1}$-valued transverse cocycle $\alpha \in\twis$ defines a homology class $[\alpha] \in H_1(\widehat\Phi; \R^{n-1})$ and a weight $\alpha(e) \in \R^{n-1}$ associated to each edge $e$ of the train track $\widehat \Phi$. In addition, $\alpha \in\twis$ is completely determined by the homology class $[\alpha] \in H_1(\widehat\Phi; \R^{n-1})$, or by the family of the edge weights  $\alpha(e) \in \R^{n-1}$, and the homology classes or edge weight systems that are thus associated to transverse cocycles are easily characterized. See \S \ref{sect: twisted cocycles} and \S \ref{sect: transverse cocycles and homology} for details. 

For $a=1$, $2$, \dots, $n-1$, let $[\alpha^{(a)}]  \in H_1(\widehat\Phi; \R)$ denote the $a$th component of  $[\alpha] \in H_1(\widehat\Phi; \R^{n-1})$, and let $\alpha^{(a)} (e)\in \R$ be the $a$th coordinate of the edge weight $\alpha(e) \in \R^{n-1}$.

The first formulation of our main result is homological.

\begin{thm}
\label{thm:MainThm1}
Let S be a closed oriented surface with genus $g\geq 2$ and let $\lambda$ be a maximal geodesic lamination carried by a train track $\Phi$ in $S$.  If the  vectors $U_{\alpha_1}$, $U_{\alpha_2}\in T_\rho \Hit_n(S)$ are tangent to the shearing deformations of $\rho \in \Hit_n(S)$ along $\lambda$ respectively associated to the transverse cocycles
$\alpha_1$, $\alpha_2 \in \mathcal Z(\lambda; \widetilde \R^{n-1})$, then for the Atiyah-Bott-Goldman symplectic form $\omega$, 
$$
\omega(U_{\alpha_1},U_{\alpha_2}) = \sum_{a,b=1}^{n-1}  C(a,b) [\alpha_1^{(a)}] \cdot [\alpha_2^{(b)}]
$$
where $[\alpha_1^{(a)}] \cdot [\alpha_2^{(b)}] \in \R$ denotes the algebraic intersection number of the homology classes $[\alpha_1^{(a)}]$, $[\alpha_2^{(a)}]  \in H_1(\widehat\Phi; \R)$ in the oriented surface $\widehat\Phi$, and where
$$
C(a,b)=
\begin{cases}
  2a(n-b)& \mathrm{ if } \, a\leq b\\
  2b(n-a)& \mathrm{ if } \, a\geq b .
\end{cases}
$$
\end{thm}

Theorem~\ref{thm:MainThm1} is proved as Theorem~\ref{maintheorem1} in \S \ref{sect:Proof main thms}. 
A simple computation (see Lemma~\ref{lem:ComputeIntersectionEdgeWeightSystems}) rephrases this result with a more explicit formula, in terms of edge weights. 

\begin{thm}
\label{thm:MainThm2}
Under the hypotheses of Theorem~{\upshape\ref{thm:MainThm1}}, $\omega(U_{\alpha_1},U_{\alpha_2})$ is equal to 
$$
{ \frac{1}{2}}  \sum_{a, b=1}^{n-1} \sum_{s} C(a,b)  \bigl( \alpha_1^{(a)}(e_s^{\mathrm{right}}) \,\alpha_2^{(b)}(e_s^{\mathrm{left}}) -\alpha_1^{(a)} (e_s^{\mathrm{left}})\, \alpha_2^{(b)} (e_s^{\mathrm{right}})\bigr)
$$
where the third sum is over all switches of the train track $\widehat{\Phi}$ and where, at each such switch $s$, $e_s^{\mathrm{right}}$ and $e_s^{\mathrm{left}}$ denote the two edges of $\widehat\Phi$ outgoing from $s$ on the right and on the left, respectively.
\end{thm}

Theorems~\ref{thm:MainThm1} and \ref{thm:MainThm2} are the natural extensions of an earlier computation of $\omega$ by S\"{o}zen-Bonahon \cite{SozBon} for the case $n=2$. In that case, every vector $u\in T_\rho \Hit_2(S)$ is tangent to a shearing deformation along $\lambda$, so that their result provides a full computation of the Atiyah-Bott-Goldman symplectic form $\omega$ for the  Teichm\"uller space $\Hit_2(S)$, in terms of the Thurston parametrization of $\Hit_2(S)$ by the shear coordinates associated to $\lambda$ \cite{Th3, Bon1}.

\subsection*{Acknowledgements}
This work was carried out while the author was visiting the University of Southern California,  with funding provided by the  Scientific and Technological
Research Council of Turkey (T\"{U}B\.{I}TAK). She is grateful to Professor Francis Bonahon for his endless support, and  to her supervisor Professor Ya\c{s}ar S\"{o}zen for giving her this opportunity. Finally, she owes a debt of gratitude to her husband Halil Zeybek and her daughter Huzur Zeybek, who always stood by her in this process.

\section{Geodesic laminations and Transverse Cocycles}
We recall a few definitions on geodesic laminations and transverse cocycles. For more detail, see \cite{Bon1, Bon2}.

\subsection{Geodesic laminations}
\label{sect: geod laminations}
Throughout the article, we fix the choice of a Riemannian metric $m$ with negative curvature on the surface $S$. A \emph{geodesic lamination} is a closed subset $\lambda\subset S$ that can be decomposed as a disjoint union of simple complete $m-$geodesics, called its \emph{leaves}. Recall that a geodesic is \emph{complete} if it cannot be extended to a longer  geodesic and it is \emph{simple} if it has no transverse self-intersection point.  The leaves of a geodesic lamination can be closed or bi-infinite. A geodesic lamination can have, either finitely many leaves, or uncountably many leaves. For basic facts about geodesic laminations, we refer \cite{Th1, CEG, PeH}.

In particular, the definition of a geodesic lamination can be made independent of our choice of a negatively curved metric $m$, and our constructions will be independent of this choice. 

A geodesic lamination $\lambda$ is \emph{maximal}, if it is not contained in any larger geodesic lamination. This is equivalent to the property that the complement $S- \lambda$ consists of finitely many disjoint ideal triangles, namely triangles isometric to an infinite triangle in hyperbolic space $\mathbb{H}^2$ whose three vertices are on the circle at infinity.

\subsection{$\mathbb{R}$-valued transverse cocycles}
\label{sect: R-valued transverse cocycles}
An \emph{$\mathbb{R}$-valued transverse cocycle} $\alpha$ for the geodesic lamination $\lambda$ can be thought as a transverse finitely additive signed measure for $\lambda$. More precisely, $\alpha$ assigns a number $\alpha(k) \in \mathbb{R}$ to each arc $k \subset S$ which is transverse to (the leaves of) $\lambda$ and which satisfies the following two conditions:
\begin{enumerate}
\item $\alpha$ is finitely additive, namely $\alpha(k)=\alpha(k_1)+\alpha(k_2)$ whenever the arc $k$ transverse to $\lambda$ is decomposed into two arcs $k_1$, $k_2$ with disjoint interiors;
\item $\alpha$ is invariant under homotopy respecting $\lambda$, in the sense that  $\alpha(k)=\alpha(k')$ whenever the arcs $k$ and $k'$ are homotopic through a family of arcs which are all transverse to $\lambda$.
\end{enumerate}
By convention, the endpoints of an arc transverse to $\lambda$ are assumed to be in the complement $S- \lambda$.

The $\mathbb{R}$-valued transverse cocycles for the geodesic lamination $\lambda$ form a vector space  $\mathcal{Z}(\lambda;\mathbb{R})$.

A (trivalent) \emph{train track} $\Phi$ in the surface $S$ is the union of finitely many ``long'' rectangles $e_i$ which are foliated by arcs parallel to the ``short'' sides. These long rectangles meet only along arcs (possibly reduced to a point) contained in their short sides and satisfy the following properties:
\begin{enumerate}
 \item Each short side of a rectangle $e_i$ is either contained in a short side of a rectangle $e_j$, or is the union of two short sides of rectangles $e_j$ and $e_k$ meeting in one point (where $e_i$, $e_j$, and $e_k$ are not necessarily distinct); a point where three such rectangles $e_i$, $e_j$, $e_k$ meet then forms a ``spike'' in the closure of the complement $S- \Phi$.
 \item No component of the closure of $S-\Phi$ is a disk with $0$, $1$, or $2$ spikes, or an annulus with no spike. 
\end{enumerate}

The rectangles $e_i$ are the \emph{edges} of $\Phi$. The leaves of the foliation of $\Phi$ are the \emph{ties} of the train track. The finitely many ties where three edges meet are the \emph{switches} of the train track $\Phi$. A tie that is not a switch is called \emph{generic}. Note that if we shrink each tie to a point, then the train track $\Phi$ collapses to a trivalent graph. 

A geodesic lamination $\lambda$ is \emph{carried} by the train track $\Phi$, if it is contained in the interior of $\Phi$, and  if each leaf of $\lambda$ is transverse to the ties of $\Phi$. Every geodesic lamination is carried by some train track; see for instance  \cite{Th1, PeH, CEG}.

An  \emph{edge weight system} for the train track $\Phi$  assigns a weight $a(e)\in \mathbb{R}$ to each edge $e$ of $\Phi$ in such a way that for each switch $s$ that is adjacent to the edge $e_i$ on one side and to the edges $e_j$ and $e_k$ on the other side, the following \emph{switch relation}
$$
a(e_i) = a(e_j) + a(e_k)
$$
holds. 

We let $\mathcal{W}(\Phi;\mathbb{R})$ denote the vector space of all edge weight systems for $\Phi$. 

A fundamental example of edge weight system  $a_\alpha \in \mathcal{W}(\Phi;\mathbb{R})$  arises from a transverse cocycle $\alpha \in \mathcal{Z}(\lambda;\mathbb{R})$, when the geodesic lamination $\lambda$ is carried by the train track $\Phi$. Indeed, for every edge $e$ of $\Phi$ and for an arbitrary tie $k_e$ of $e$, define $a_\alpha(e)=\alpha{(k_e)}$. The invariance of $\alpha$ under homotopy respecting $\lambda$ implies that $\alpha(k_e)$ is independent of the tie $k_e$, and the switch relation follows from the finite additivity of $\alpha$. As a consequence, the edge weights $a_\alpha(e)$ define an element $a_\alpha$ of  $\mathcal{W}(\Phi;\mathbb{R})$. 

\begin{prop}
\label{prop:transversecocyclesweightinR}
Let $\lambda$ be a maximal geodesic lamination carried by the train track $\Phi$. Then, the above map $\alpha \mapsto a_\alpha$ defines a linear isomorphism between the vector space $ \mathcal{Z}(\lambda;\mathbb{R})$ of all $\R$-valued  transverse cocycles for $\lambda$ and the vector space $\mathcal{W}(\Phi;\mathbb{R})$ of all edge weight systems for $\Phi$. 

In addition,  these two vector spaces are isomorphic to $\mathbb{R}^{3|\mathcal{X}(S)|}$, where $\mathcal{X}(S)$ is the Euler characteristic of $S$. 
\end{prop}
\begin{proof}
See \cite[Theorems 11 and 15]{Bon2}.
\end{proof}

\subsection{Twisted $\mathbb{R}^{n-1}$-valued transverse cocycles}
\label{sect: twisted cocycles}

 The notion of $\R$-valued transverse cocycle for a geodesic lamination straightforwardly generalizes to $\R^{n-1}$-valued transverse cocycles. So far, the transverse arcs that we considered in the definition of transverse cocycles were not oriented. However, as in \cite{Dr, BonDr2}, we will need a version of $\mathbb{R}^{n-1}$-valued transverse cocycles that takes the orientation of arcs into account.

A \emph{twisted $\mathbb{R}^{n-1}$-valued transverse cocycle} for $\lambda$ assigns a vector $\alpha(k) \in \mathbb{R}^{n-1} $ to each \emph{oriented} arc $k \subset S$ that is transverse to  $\lambda$, in such a way that the following three conditions are satisfied:
\begin{enumerate}
\item $\alpha$ is finitely additive, in the sense that $\alpha(k)=\alpha(k_1)+\alpha(k_2)$ whenever the oriented arc $k$ transverse to $\lambda$ is split into two oriented arcs $k_1$, $k_2$ with disjoint interiors;

\item $\alpha$ is invariant under homotopy respecting $\lambda$, in the sense that  $\alpha(k)=\alpha(k')$ whenever the arcs $k$ and $k'$ are homotopic through a family of arcs which are all transverse to $\lambda$;

\item $\alpha (\overline k) = \overline{\alpha(k)}$ for every oriented transverse arc $k$, where the oriented arc $\overline k$ is obtained by reversing the orientation of $k$ and where $x \mapsto \overline{x}$ denotes  the involution of $\mathbb{R}^{n-1}$ that reverses the order of the coordinates, namely that associates $\overline{x}=(x_{n-1},x_{n-2},\ldots,x_1)$ to $x=(x_1,x_2,\ldots,x_{n-1})\in \mathbb{R}^{n-1}.$ 
\end{enumerate}

 Let $\twis$ denote the vector space of all twisted $\mathbb{R}^{n-1}$-valued transverse cocycles for $\lambda$.
  
 To understand the space $\twis$ and its elements, it is again useful to consider a train track $\Phi$ carrying $\lambda$. In general, it is not possible to orient the ties of $\Phi$ in a continuous way but $\Phi$ has a well-defined 2-fold \emph{orientation cover} $\widehat \Phi$ consisting of all pairs $(x,o)$ where $x\in \Phi$ and $o$ is an orientation of the tie $k_x$ of $\Phi$ passing through $x$. This cover $\widehat\Phi$ is a train track, and its ties are canonically oriented by using the orientation $o$ at the point $(x,o) \in \widehat\Phi$. 
 
 Let $\iota \colon \widehat\Phi \to \widehat\Phi$ be the involution that exchanges the two sheets of the covering $\widehat\Phi \to \Phi$. By construction, $\iota$ reverses the canonical orientation of the ties of $\widehat\Phi$.

For a twisted cocycle  $\alpha \in \twis$, an arbitrary tie $k_e$ of an edge $e$ of  $\widehat\Phi$ projects to an arc that is transverse to $\lambda$ in $S$ and is oriented by the canonical orientation of the ties of $\widehat \Phi$. This defines a weight $a_\alpha(e) = \alpha(k_e) \in \R^{n-1}$, which does not depend on the choice of the tie $k_e$ by invariance of $\alpha$ under homotopy respecting $\lambda$. The finite additivity property of $\alpha$ implies that these edge weights satisfy the switch relation at the switches of $\widehat\Phi$ and therefore define an $\R^{n-1}$-valued edge weight system $a_\alpha \in \mathcal W(\widehat\Phi; \R^{n-1})$. 

\begin{prop}
\label{prop:transversecocyclesweightinR(n-1)}
Let $\lambda$ be a maximal geodesic lamination carried by the train track $\Phi$, and let $\widehat\Phi$ be the orientation cover of $\Phi$. Then, the above map $\alpha \mapsto a_\alpha$ defines a linear isomorphism between the vector space $\twis$ of all twisted $\R^{n-1}$-valued  transverse cocycles for $\lambda$ and the linear subspace of $\mathcal W(\widehat\Phi; \R^{n-1})$ consisting of those edge weight systems $a \in \mathcal W(\widehat\Phi; \R^{n-1})$ such that
$$
a \big( \iota(e) \big) = \overline{a(e)}
$$
for every edge $e$ of $\widehat\Phi$. 

In addition, the dimension of $\twis$ is equal to $6(g-1) (n-\nolinebreak 1)+\lfloor \frac{1}{2}(n-1)\rfloor$, where $\lfloor x \rfloor$ denotes the largest integer that is less than or equal to $x$.
\end{prop}
\begin{proof}
See \cite[Proposition 4.7]{BonDr2}.
\end{proof}

\subsection{Edge weight systems and homology} 
\label{sect: transverse cocycles and homology}

Let $\widehat \Phi$ be the orientation cover of the train track $\Phi$. Recall that the ties of $\widehat \Phi$ are canonically oriented, so that each such tie defines a relative homology class $[k] \in H_1(\widehat\Phi, \partial \widehat\Phi; \R)$. 

Also, as a surface the train track $\widehat \Phi$ can be oriented by lifting the orientation of the surface $S$. This defines an algebraic intersection form
$$
\cdot\ \colon H_1(\widehat\Phi, \partial \widehat\Phi; \R) \times H_1(\widehat\Phi; \R) \to \R.
$$

\begin{lem}
\label{lem:TransverseCocycleDefinesHomology}
 Let $a \in \mathcal W(\widehat\Phi; \R)$ be an edge weight system for $\widehat\Phi$. Then, there is a unique homology class $[a] \in H_1(\widehat\Phi; \R)$ such that for every edge $e$ of $\widehat\Phi$ and every generic tie $k_e$ of $e$, the algebraic intersection number $ [k_e] \cdot [a]$ is equal to the weight $a(e)$.
 \end{lem}
 
\begin{proof}
 If we collapse each tie of $\widehat \Phi$ to a point, then $\widehat \Phi$ collapses to a trivalent graph $\Gamma$. Orient the edges of $\Gamma$ so that they point to the left of the oriented ties of $\widehat \Phi$.
 
 By construction, each edge $e$ of $\widehat\Phi$ projects to an edge $f_e$ of $\Gamma$. The fact that  $a \in \mathcal W(\widehat\Phi; \R)$ satisfies the switch relation implies that  the chain $\sum_f a(e) f_e$ is closed and defines a homology class $[a] \in H_1(\Gamma; \R) \cong H_1(\widehat\Phi; \R)$.
 
 The construction is specially designed so that  $ [k_e] \cdot [a] = a(e)$ for every edge $e$ of $\widehat\Phi$. It easily follows from Poincar\'e duality (or more elementary considerations) that $[a]$ is the only homology class with this property. 
 \end{proof}
 
 In fact, Lemma~\ref{lem:TransverseCocycleDefinesHomology} played a key role in the dimension computations of Propositions~\ref{prop:transversecocyclesweightinR} and \ref{prop:transversecocyclesweightinR(n-1)} in the articles \cite{Bon2, BonDr2}. 
 
 The connection between Theorems~\ref{thm:MainThm1} and \ref{thm:MainThm2} is provided by the following elementary computation. 
 At each switch $s$ of $\widehat \Phi$, there is a single incoming edge $e^{\mathrm{in}}_s$ on one side of $s$ and two outgoing edges $e^{\mathrm{left}}_s$ and $e^{\mathrm{right}}_s$ on the other side, with  $e^{\mathrm{left}}_s$ and $e^{\mathrm{right}}_s$ respectively diverging to the left and to the right as seen from the incoming edge $e^{\mathrm{in}}_s$ and for the orientation of $S$.
 
 \begin{lem}
 \label{lem:ComputeIntersectionEdgeWeightSystems}
 For two edge weight systems $a_1$, $a_2  \in \mathcal W(\widehat\Phi; \R)$, the algebraic intersection number of the associated homology classes $[a_1]$, $[a_2] \in H_1(\widehat\Phi; \R)$ is equal to
 $$
 [a_1] \cdot [a_2] =\frac12 \sum_s \big( a_1(e^{\mathrm{right}}_s)a_2(e^{\mathrm{left}}_s) - a_1(e^{\mathrm{left}}_s) a_2(e^{\mathrm{right}}_s) \big)
 $$
\end{lem}

\begin{proof}
 The graph $\Gamma$ used in the proof of Lemma~\ref{lem:TransverseCocycleDefinesHomology} can easily be chosen so that it is contained in the interior of $\widehat\Phi$, it is transverse to the ties of $\widehat\Phi$, and the inclusion map $\Gamma \to \widehat\Phi$ is a homotopy  equivalence. Let $\Gamma'$ be obtained by slightly pushing $\Gamma$ in the direction given by the orientation of the ties of $\widehat\Phi$. To compute the intersection number $ [a_1] \cdot [a_2]$, realize the homology class $[a_1]$ by a linear combination of edges of $\Gamma$, and $[a_2]$ by a linear combination of edges of $\Gamma'$. 
 
 Each switch $s$ of $\widehat\Phi$ contributes a point to the intersection $\Gamma\cap \Gamma'$. Note that there are two types of switches of $s$: those where the two outgoing edges  $e^{\mathrm{left}}_s$ and $e^{\mathrm{right}}_s$ are on the left of the oriented arc $s$, and those where they are on the right. If we refer to these two types as \emph{left-diverging} and \emph{right-diverging}, respectively, and consider their respective contributions to $ [a_1] \cdot [a_2] $, we immediately see that
  $$
 [a_1] \cdot [a_2] = \kern -10pt \sum_{s\text{ left-diverging}}  \kern -10pt a_1(e^{\mathrm{right}}_s)a_2(e^{\mathrm{left}}_s) -  \kern -10pt \sum_{s\text{ right-diverging}}  \kern -10pt a_1(e^{\mathrm{left}}_s) a_2(e^{\mathrm{right}}_s) .
 $$

Using the antisymmetry of the intersection number and writing $ [a_1] \cdot [a_2] = \frac12 \big(  [a_1] \cdot [a_2] -  [a_2] \cdot [a_1] \big)$ provides the formula stated in the lemma. 
\end{proof}

\section{Shearing deformations of Hitchin representations}
\label{sect:ShearingDef}

Shearing deformations in the Teichm\"uller space $\Hit_2(S)$ were introduced by W. Thurston \cite{Th3, Bon1}, and themselves generalize the earlier notion of earthquakes \cite{Th2}. In that case, a shearing deformation is determined by a transverse cocycle $\alpha \in \mathcal Z( \lambda;\R)$ and, when $\lambda$ is a maximal geodesic lamination, provide a parametrization of $\Hit_2(S)$ by an explicit polytope in $\mathcal Z( \lambda;\R)$. 

The extension of shearing deformations to the Hitchin component $\Hit_n(S)$ was developed by Dreyer and Bonahon in \cite{Dr, BonDr2}. We now review the main features of this construction. 

\subsection{The Hitchin component} 
\label{sect: Hitchin}
 Let $S$ be a closed, connected, oriented surface of genus at least 2. The character variety of $S$ is defined to be  
\[\mathcal{X}_{\mathrm{PSL}_n(\mathbb{R})}(S)=\{\mathrm{homomorphisms} \,\, \rho : \pi_1(S)\rightarrow \mathrm{PSL}_n(\mathbb{R})\} /\kern-3pt/ \mathrm{PSL}_n(\mathbb{R}) \]
 consisting of all group homomorphisms $\pi_1(S)\rightarrow \mathrm{PSL}_n(\mathbb{R})$ from the fundamental group $\pi_1(S)$ to the Lie group $\mathrm{PSL}_n(\mathbb{R})$, considered up to conjugation by elements of  $\mathrm{PSL}_n(\mathbb{R})$. Note that the projective special linear group $\mathrm{PSL}_n(\R)$ is equal to the special linear group $\mathrm{SL}_n(\R)$ if $n$ is odd, and to $\mathrm{SL}_n(\R)/\{\pm Id\}$ if $n$ is even.
 
In the case where $n=2$, there is a preferred component of $\mathcal{X}_{\mathrm{PSL}_2(\mathbb{R})}(S)$, which is known as \emph{Teichmüller component} $\Hit_2(S)$. As is well known that the Teichmüller component $\Hit_2(S)$ is diffeomorphic to the space of complex structures on $S$ by the Uniformization Theorem.
 
 In the general case, there is a preferred homomorphism $\mathrm{PSL}_2(\mathbb{R}) \rightarrow \mathrm{PSL}_n(\mathbb{R})$ coming from the unique $n-$dimensional irreducible representation of $\mathrm{SL}_2(\mathbb{R}).$ This provides a natural map from $\mathcal{X}_{\mathrm{PSL}_2(\mathbb{R})}(S)$ to $\mathcal{X}_{\mathrm{PSL}_n(\mathbb{R})}(S).$ 
  
  The \emph{Hitchin component} Hit$_n(S)$ is the component of $\mathcal{X}_{\mathrm{PSL}_n(\mathbb{R})}(S)$ that contains the image of the Teichm\"{u}ller component of $\mathcal{X}_{\mathrm{PSL}_2(\mathbb{R})}(S).$ N. Hitchin \cite{Hi} was the first to single out this component. 
 
 The elements of the Hitchin component $\Hit_n(S)$ similarly enjoy many powerful geometric properties  from a variety of viewpoints, involving differential geometry, dynamical systems or the theory of Lie groups; see for instance the fundamental results of  Hitchin \cite{Hi}, Labourie \cite{Lab} and Fock-Goncharov \cite{FoG}. In particular, using the point of view of Higgs bundles, Hitchin proved in \cite{Hi} that $\Hit_n(S)$ is diffeomorphic to the Euclidean space $\R^{2(g-1)(n^2-1)}$.
 
 A \emph{Hitchin character} is an element of the Hitchin component, and a \emph{Hitchin homomorphism} is a homomorphism $\rho \colon \pi_1(S) \to \mathrm{PSL}_n(\R)$ representing a Hitchin character.

\subsection{The flag curve}
\label{sect: flag curve}

A \emph{flag} in $\mathbb{R}^n$ is a family $F$ of nested linear subspaces $F^{(0)}\subset F^{(1)}\subset \ldots \subset F^{(n-1)}\subset F^{(n)}$ of $\mathbb{R}^n$ where each $F^{(a)}$ has dimension $a.$

Two flags $E$ and $F$ are \emph{transverse to each other}, if every subspace $E^{(a)}$ of $E$ is transverse to every subspace $F^{(b)}$ of $F.$ This is equivalent to property that $\mathbb{R}^n=E^{(a)}\oplus F^{(n-a)}$ for every $a$. 

Let $\widetilde S$ be the universal cover of $S$, and let $\partial_\infty \widetilde S$ be its circle at infinity. 

\begin{prop}[\cite{Lab}]
\label{prop:flag curve}
For a Hitchin representation
$\rho \colon  \pi_1(S) \rightarrow \mathrm{PSL}_n(\mathbb{R}),$
there exists a unique continous map 
$$\mathcal{F}_{\rho}\colon \partial_{\infty}\widetilde{S}\rightarrow \mathrm{Flag}(\mathbb{R}^n)$$ that is  $\rho$--equivariant in the sense that $\mathcal{F}_{\rho}(\gamma{x})=\rho(\gamma)(\mathcal{F}_{\rho}({x}))$ for every ${x} \in \partial_{\infty}\widetilde{S}$ and $\gamma \in \pi_1(S)$.

In addition, $\mathcal{F}_{\rho}$ is H\"{o}lder continuous, and for any two distinct points $ x$, $ y \in \partial_{\infty}\widetilde{S}$, the two flags $\mathcal{F}_{\rho}({x})$, $\mathcal{F}_{\rho}({y})\in \mathrm{Flag}(\mathbb{R}^n)$ are transverse to each other.
\end{prop}

 By definition, this map $\mathcal{F}_{\rho}\colon \partial_{\infty}\widetilde{S}\rightarrow \mathrm{Flag}(\mathbb{R}^n)$ is called the \emph{flag curve} of the Hitchin homomorphism $\rho\colon \pi_1(S)\rightarrow \mathrm{PSL}_n(\mathbb{R})$.
 
\subsection{Elementary shearing}
\label{sect: elementary shearing}
 Let $g$ be an oriented geodesic of the universal cover $\widetilde{S}$ with positive endpoint ${x}_+ \in \partial_{\infty}\widetilde{S}$ and negative endpoint ${x}_-\in \partial_{\infty}\widetilde{S}$. By Proposition \ref{prop:flag curve}, there exist two flags  $E=\mathcal{F}_{\rho}({x}_-)$, $F=\mathcal{F}_{\rho}({x}_+) \in \flag$ which are transverse to each other, and therefore define a  line decomposition $\mathbb{R}^n = \bigoplus_{a=1}^n L_a$ with $L_a=E^{(a)}\cap F^{(n-a+1)}$. 
 
 For $(u_1,u_2,\ldots,u_{n-1})\in \mathbb{R}^{n-1}$, let $v_1$, $v_2$, \dots, $v_n\in \R$ be uniquely determined by the properties that $u_a=v_a-v_{a+1}$ and $\sum_{a=1}^n v_a=0$.  Then, let $T_g^{(u_1,u_2,\ldots,u_{n-1})}\colon \mathbb{R}^n \rightarrow \mathbb{R}^n$ be the linear map that respects the line decomposition $\mathbb{R}^n = \bigoplus_{a=1}^n L_a$ and  acts as multiplication by $e^{v_a}$ on each line $L_a$.
 
 This definition can be made more explicit by noting that $v_a = -\sum_{b=1}^{a-1} \frac bn u_b + \sum_{b=a}^{n-1} \frac{n-b}n u_b$ but this is not necessarily more enlightening.

 Note that $T_g^{(u_1,u_2,\dots,u_{n-1})}$ is an element of $\SL$. By definition, $T_g^{(u_1,u_2,\ldots,u_{n-1})}$ is the \emph{elementary shearing map of amplitude $(u_1,u_2,\dots,\\u_{n-1})$ along the geodesic $g$}. 
 
 To motivate this definition, it is useful to consider the double ratio $D_a(E,F,G,H)\in \R$ of a suitably  generic quadruple of flags $(E,F,G,H)\\ \in \flag^4$, which is the crucial ingredient in the definition of the shearing coordinates associated to positive framed representations by Fock-Goncharov \cite{FoG} for punctured surfaces, as well as of the shearing cycle of the Bonahon-Dreyer parametrization of $\Hit_n(S)$ for closed surfaces \cite{BonDr1, BonDr2}. An easy computation and the definition yield the following property.

\begin{lem} 
If $E=\mathcal{F}_{\rho}({x}_-)$ and $F=\mathcal{F}_{\rho}({x}_+) \in \flag$ are the flags associated to the geodesic $g$, and if $(E,F,G,H) \in \flag^4$ is sufficiently generic that the double ratio $D_a(E,F,G,H)$ is defined, then
$$
  D_a(E,F,G,T_g^{(u_1,\dots ,u_{n-1})}H)= e^{u_a} D_a(E,F,G,H)
$$ 
for every $a=1$,  \dots, $n-1$. \qed
\end{lem}

By construction, the shearing map $T_g^{(u_1,u_2,\ldots,u_{n-1})} \in \SL$ is the identity when $(u_1,u_2,\ldots,u_{n-1})=(0,0, \dots, 0)$. Let us denote the \emph{infinitesimal $a$-shearing map along $g$} by $t_g^{(a)}$ which is the derivative of $T_g^{(u_1,u_2,\ldots,u_{n-1})}$ with respect to the $a$th coordinate $u_a$ at that point. Namely,
$$
t_g^{(a)} = {\textstyle\frac d{du}} T_g^{(0,\dots,0, u, 0, \dots, 0)} {}_{|u=0}
$$
as an element of the Lie algebra $\mathfrak{sl}_n(\mathbb{R})$, where the variable $u$ occurs in the $a$-th position in  $(0,\dots,0, u, 0, \dots, 0)$. In other words, $t_g^{(a)} $ is the linear map $\R^n \to \R^n$ that respects the line decomposition   $\mathbb{R}^n = \bigoplus_{b=1}^n L_b$ with $L_b=E^{(b)}\cap F^{(n-b+1)}$ and acts on $L_b$ by multiplication by  $\frac{n-a}n$ if $b\leq a$ and by $-\frac an$ if $b \geq a+1$. 
 

\subsection{Shearing a Hitchin representation} 
\label{subsect:ShearingDef}

Let $\lambda$ be a maximal geodesic lamination in $S$. 
The shearing deformation $\Sigma^\alpha \rho$ of the Hitchin character $\rho \in \Hit_n(S)$ is determined by a twisted $\R^{n-1}$-valued transverse cocycle $\alpha \in \twis$. 

To construct  $\Sigma^\alpha \rho$, we begin for describing the shearing between two components $P$, $Q$ of the complement $\widetilde S - \widetilde\lambda$ of the preimage $\widetilde \lambda$ of $\lambda$ in the universal cover $\widetilde S$.  Note that $P$ and $Q$ are ideal triangles.

Let $\mathcal{P}_{PQ}$ be the set of components of $\widetilde{S}-\widetilde{\lambda}$ between $P$ and $Q$,  let $k \subset \widetilde S$ be an oriented arc that goes from a point in the interior of $P$ to a point in the interior of $Q$, is transverse to $\widetilde\lambda$, and is non-backtracking in the sense that it meets each leaf of $\widetilde\lambda$ at most once. 

Let $\mathcal{P}=\{R_1,R_2,\ldots,R_m\}$ be a finite subset of $\mathcal{P}_{PQ}$, where $R_1$, $R_2$, \dots, $R_m$ occur in this order as one goes from $P$ to $Q$.  

For every ideal triangle $R_i\in \mathcal{P}$, let $g_i^-$ and $g_i^+\subset \widetilde{\lambda}$ be the two leaves bounding $R_i$ that are the closest to the triangles $P$ and $Q$, respectively. Orient the two geodesics $g_i^\pm$ to the left of the oriented arc $k$, and let $\alpha(k_i) \in \R^{n-1}$ be the vector associated by $\alpha \in \twis$ to the projection to $S$ of an arbitrary subarc $k_i$ of $k$ joining a point in the interior of $P$ to a point in the interior of $R_i$. We can associate to the oriented geodesics $g_i^-$ and $g_i^+$ the elementary shearing maps $T_{g_i^-}^{\alpha(k_i)}$ and $T_{g_i^+}^{-\alpha(k_i)} \in \SL$.

With the above definitions, we can then consider the linear map
\begin{multline*}
\varphi_\mathcal{P}^\alpha=T_{g_1^-}^{\alpha(k_1)}\circ T_{g_1^+}^{-\alpha(k_1)}\circ T_{g_2^-}^{\alpha(k_2)}\circ T_{g_2^+}^{-\alpha(k_2)}\circ \dots\\
\dots \circ T_{g_m^-}^{\alpha(k_m)}\circ T_{g_m^+}^{-\alpha(k_m)}\circ T_{g_Q^-}^{\alpha(k)}.
\end{multline*}

\begin{lem}
If $\alpha$ is sufficiently close to $0$ in $\twis$,  the above linear map $\varphi_{\mathcal P}^\alpha$ converges to a  map $\varphi_{PQ}^\alpha$ in $ \mathrm{SL}_n(\mathbb{R})$ as the finite subset $\mathcal P$ converges to the set $\mathcal{P}_{PQ}$ of all components of $\widetilde S - \widetilde \lambda$ separating $P$ from $Q$.  
\end{lem}
\begin{proof}
 See \cite[Lemma 24]{Dr} or \cite[\S8.3]{BonDr2}. 
\end{proof}

Fix a component $P_0$ of the complement $\widetilde S- \widetilde\lambda$. 

\begin{lem}
If $\alpha$ is sufficiently close to $0$ in $\twis$,  there exists a unique group homomorphism $\Sigma^\alpha \rho \colon \pi_1(S) \to \psl$ such that
$$
\Sigma^\alpha \rho (\gamma) = \varphi_{P_0(\gamma P_0)}^\alpha \, \rho(\gamma) \in \psl
$$
 for every $\gamma \in \pi_1(S)$. In addition, $\Sigma^\alpha \rho$ is a Hitchin representation. \end{lem}

\begin{proof}
See \cite[Lemma 26]{Dr} and \cite[\S8.3]{BonDr2}. 
\end{proof}

By definition, $\Sigma^\alpha \rho \in \Hit_n(S)$ is obtained by \emph{shearing $\rho$ along the geodesic lamination $\lambda$ according to the twisted transverse cocycle $\alpha \in \twis$}. The group homomorphism $\Sigma^\alpha \rho \colon \pi_1(S) \to \psl$ depends on the choice of the base component $P_0$ of $\widetilde S- \widetilde \lambda$ but the Hitchin character $\Sigma^\alpha \rho \in \Hit_n(S)$ that it represents does not. 

In the Bonahon-Dreyer parametrization \cite{BonDr2} of $\Hit_n(S)$ associated to the maximal geodesic lamination $\lambda$, a Hitchin representation $\rho$ is parametrized by certain triangle invariants $\tau_{abc}^\rho (s) \in \R$ and by a shearing relative cocycle $\sigma^\rho \in \mathcal Z_{tw}(\lambda, \mathrm{slits};  \R^{n-1})$. It follows from the definitions that the Hitchin character $\rho'=\Sigma^\alpha \rho$ obtained by shearing $\rho \in \Hit_n(S)$ according to $\alpha \in \twis \subset  \mathcal Z_{tw}(\lambda, \mathrm{slits};  \R^{n-1})$ has the same triangle invariants $\tau_{abc}^{\rho'}(s) = \tau_{abc}^\rho (s)$ as $\rho$ but that its shearing cocycle is equal to $\sigma^{\rho'} = \sigma^\rho + \alpha$.

\subsection{Infinitesimal shearing of a Hitchin representation}
\label{sect: infinitesimal shearing}

The above construction associates a tangent vector
$$
U_\alpha = {\textstyle \frac{d}{dt}} \Sigma^{t\alpha} \rho_{|t=0} \in T_\rho \Hit_n(S)
$$
to each Hitchin character $\rho \in \Hit_n(S)$ and each twisted transverse cocycle $\alpha \in \twis$. By definition,  the tangent vector $U_\alpha \in T_\rho \Hit_n(S)$ is the \emph{infinitesimal shearing vector} of $\rho \in \Hit_n(S)$ along the geodesic lamination $\lambda$ according to the twisted transverse cocycle $\alpha \in \twis$. 

The standard deformation theory of representations  \cite{We2} identifies the tangent space $T_{\rho}\mathrm{Hit}_n(S)$ to $H^1(S;\mathfrak{sl}_n(\mathbb{R})_{\mathrm{Ad}_{\rho}})$ the first cohomology space   of the surface $S$ with coefficients in the Lie algebra $\mathfrak{sl}_n(\mathbb{R})$ of $\mathrm{PSL}_n(\mathbb{R})$ twisted by the adjoint representation $\mathrm{Ad}_{\rho}$ from $ \pi_1(S)$ to $ \mathrm{Aut}(\mathfrak{sl}_n(\mathbb{R}))$. 

Before describing the cohomology class of $H^1(S;\mathfrak{sl}_n(\mathbb{R})_{\mathrm{Ad}_{\rho}})$ corresponding to the tangent vector $U_\alpha \in T_\rho \Hit_n(S)$, we recall the definition of this cohomology space. First of all, the adjoint representation is defined by the property that for every  $\gamma \in \pi_1(S)$, the automorphism $\mathrm{Ad}_\rho (\gamma) \colon 
\mathfrak{sl}_n(\mathbb{R}) \to \mathfrak{sl}_n(\mathbb{R})$ is the matrix conjugation automorphism $u\mapsto \rho(\gamma)u\rho(\gamma)^{-1}$ induced by $\rho(\gamma) \in \psl$. We then consider the vector space
$
C^k(S;\mathfrak{sl}_n(\mathbb{R})_{\mathrm{Ad}_{\rho}})$
consisting of all $\pi_1(S)$-equivariant group homomorphisms $C_k(\widetilde{S};\mathbb{Z})\rightarrow \mathfrak{sl}_n(\mathbb{R})$, where $\pi_1(S)$ acts on the space $C_k(\widetilde{S};\mathbb{Z})$ of $k$-chains by its action on the universal cover $\widetilde S$ and acts on $\mathfrak{sl}_n(\mathbb{R})$ by the adjoint representation $\mathrm{Ad}_\rho$. The boundary maps of the chain complex $C_*(\widetilde{S};\mathbb{Z})$ then induce coboundary maps turning the family of these vector spaces into an ascending chain complex $C^*(S;\mathfrak{sl}_n(\mathbb{R})_{\mathrm{Ad}_{\rho}})$. The cohomology space $H^1(S;\mathfrak{sl}_n(\mathbb{R})_{\mathrm{Ad}_{\rho}})$ is the first homology space of this chain complex. 

The above discussion is independent of the type of (co)homology that we are considering for the universal cover $\widetilde S$. For our purposes, it will be convenient to restrict $C_1(\widetilde{S};\mathbb{Z})$ to 1-chains that are linear combinations of oriented arcs transverse to the geodesic lamination $\widetilde\lambda$, and $C_2(\widetilde{S};\mathbb{Z})$ to 2-chains whose boundary is of the above type. The standard uniqueness theorems for homology and cohomology theories show that 
this framework can be used to compute $H^1(S;\mathfrak{sl}_n(\mathbb{R})_{\mathrm{Ad}_{\rho}})$. 

We also note that that Weil's original construction \cite{We2} of the isomorphism $ T_\rho \Hit_n(S)\cong H^1(S;\mathfrak{sl}_n(\mathbb{R})_{\mathrm{Ad}_{\rho}})$ passes through a similarly defined group cohomology space $H^1(\pi_1(S);\mathfrak{sl}_n(\mathbb{R})_{\mathrm{Ad}_{\rho}})$, which is isomorphic to $ H^1(S;\mathfrak{sl}_n(\mathbb{R})_{\mathrm{Ad}_{\rho}})$ as $\widetilde S$ is contractible.

\begin{lem}[Gap Formula]
\label{gap formula}
If $U_\alpha \in T_{\rho}\mathrm{Hit}_n(S)$ is the infinitesimal shearing vector along the geodesic lamination $\lambda$ according to the twisted transverse cocycle $\alpha \in \twis$, then the cohomology class of $H^1(S;\mathfrak{sl}_n(\mathbb{R})_{\mathrm{Ad}_{\rho}})$ corresponding to $U_\alpha$ can be realized by the closed cochain  $u_\alpha \in C^1(S; \mathfrak{sl}_n(\mathbb{R})_{\mathrm{Ad}_{\rho}})$ such that for every oriented arc ${k}$ transverse to $\widetilde{\lambda}$
\begin{equation*}
u_\alpha({k})=
\sum_{a=1}^{n-1} \alpha^{(a)}({k})t_{g_{d_+}^{-}}^{(a)} 
+ \sum_{a=1}^{n-1} \sum_{d\neq d_+,d_-} \alpha^{(a)} ( {k_d}) \big( t_{g_d^{-}}^{(a)}-t_{g_d^{+}}^{(a)} \big), 
\end{equation*}
where: 
\begin{itemize}
\item for every subarc $k'$ of $k$, $\alpha^{(a)}(k')\in \R$ is the $a$th coordinate of the vector $\alpha(k) \in \R^{n-1}$ associated by $\alpha$ to the projection of $k'$ to $S$;
\item $t_g^{(a)} \in \mathfrak{sl}_n(\mathbb{R})$ is the infinitesimal $a$th shearing map along the oriented geodesic $g$;
\item the last sum is over all components d of ${k}-\widetilde{\lambda}$ which are distinct from the components $d_+$ and $d_-$ respectively containing the positive and the negative end points of ${k}$;
\item ${k_d}$ is a subarc of ${k}$ that joins the negative end point of ${k}$ to an arbitrary point in $d$;
\item $g_d^+$ and $g_d^-$ are the leaves of $\widetilde{\lambda}$ respectively passing through the positive and negative end points of $d$ and are oriented to the left of $\widetilde{k}$.
\end{itemize}
\end{lem}

\begin{proof} Note that the formula usually involves infinitely many components $d$ of $k - \widetilde \lambda$. The convergence of the sum is proved by arguments similar to the ones in \cite[Section 5]{Bon1}; compare also the Gap Formula in \cite[Theorem 10]{Bon2}, as well as our proof of Lemma~\ref{gapapp} below.

To explain what the formula represents, let $P$ and $Q$ be the components of $\widetilde S - \widetilde\lambda$ that respectively contain the negative and positive endpoints of $k$. Then, considering the definition of the shearing map $\varphi^{t\alpha}_{PQ}\in \SL$ used to define the shearing deformation $\Sigma^{t\alpha} \rho \in \Hit_n(S)$ in \S \ref{subsect:ShearingDef} and differentiating in $t$, we see that 
$$
u_\alpha({k}) =\left( {\textstyle \frac d{dt}}  \varphi^{t\alpha}_{PQ} \right)_{|t=0}. 
$$

The fact that the cochain $u_\alpha \in C^1(S,\mathfrak{sl}_n(\mathbb{R})_{\mathrm{Ad}_{\rho}})$ is closed then comes from the property that $ \varphi^{t\alpha}_{PQ} \circ  \varphi^{t\alpha}_{QR} =  \varphi^{t\alpha}_{PR}$ for any three components $P$, $Q$, $R$ of $\widetilde S- \widetilde\lambda$, which is a key feature of the maps $\varphi^{t\alpha}_{PQ}  \in \mathfrak{sl}_n(\mathbb{R})$ proved in \cite[Theorem~20]{Dr}. 

By definition of the isomorphism $T_{\rho}\mathrm{Hit}_n(S) \cong H^1(\pi_1(S);\mathfrak{sl}_n(\mathbb{R})_{\mathrm{Ad}_{\rho}})$ in \cite{We2}, the group cohomology class  corresponding to $U_\alpha \in T_{\rho}\mathrm{Hit}_n(S)$ is defined by the group cocycle $u \colon \pi_1(S) \to \mathfrak{sl}_n(\mathbb{R})$ that associates to each $\gamma \in \pi_1(S)$ the derivative 
$$
u(\gamma) = \left( {\textstyle \frac d{dt}}  \big(\Sigma^{t\alpha}\rho(\gamma)\big) \, \rho(\gamma)^{-1} \right)_{|t=0} = 
\left( {\textstyle \frac d{dt}}  \varphi^{t\alpha}_{P_0(\gamma P_0) } \right)_{|t=0} = u_\alpha(k_\gamma),
$$
where $k_\gamma$ is an arbitrary arc transverse to $\widetilde \lambda$ and going from $P_0$ to $\gamma P_0$. It follows that the isomorphism $H^1(\pi_1(S);\mathfrak{sl}_n(\mathbb{R})_{\mathrm{Ad}_{\rho}}) \cong H^1(S;\mathfrak{sl}_n(\mathbb{R})_{\mathrm{Ad}_{\rho}})$ sends this group cohomology class $[u]$ to the cohomology class $[u_\alpha] \in H^1(S;\mathfrak{sl}_n(\mathbb{R})_{\mathrm{Ad}_{\rho}})$ represented by the closed cochain  $u_\alpha$. 
\end{proof}

\section{The Atiyah-Bott-Goldman symplectic pairing of two infinitesimal shearings}

\subsection{The Atiyah-Bott-Goldman symplectic pairing}
\label{sect: ABG}

The Cartan-Killing bilinear form $B\colon \mathfrak{sl}_n(\mathbb{R})\times \mathfrak{sl}_n(\mathbb{R}) \rightarrow \mathbb{R}$ is defined as $B(u,v)=2n \mathrm{Tr}(uv)$. It is preserved by the adjoint representation and therefore enables us to define a cup product
\begin{equation}
\label{eqcupprod}
\smile \colon  C^1(S;\mathfrak{sl}_n(\mathbb{R})_{\mathrm{Ad}_{\rho}})\times C^1(S;\mathfrak{sl}_n(\mathbb{R})_{\mathrm{Ad}_{\rho}})\rightarrow C^2(S,\mathbb{R}),
\end{equation}
which induces an antisymmetric bilinear form
\[
\omega \colon
H^1(S;\mathfrak{sl}_n(\mathbb{R})_{\mathrm{Ad}_{\rho}})\times H^1(S;\mathfrak{sl}_n(\mathbb{R})_{\mathrm{Ad}_{\rho}})\rightarrow H^2(S;\mathbb{R})\cong \mathbb{R}\]
where the isomorphism $H^2(S,\mathbb{R})\cong \mathbb{R}$ is defined by evaluation on the fundamental class of the oriented surface $S$.

Goldman \cite{Go} showed (under a higher level of generality) that for the isomorphism $T_{\rho}\mathrm{Hit}_n(S) \cong H^1(S;\mathfrak{sl}_n(\mathbb{R})_{\mathrm{Ad}_{\rho}})$, this form $\omega$ defines a symplectic form on $\Hit_n(S)$ now known as the \emph{Atiyah-Bott-Goldman symplectic form}. He also showed that in the case where $n=2$, this symplectic form is a constant multiple of  the Weil-Petersson symplectic form of the Teichm\"uller space $\mathcal{T}(S)= \Hit_2(S)$.

Our goal is to compute the pairing $\omega(U_{\alpha_1}, U_{\alpha_2})$ of the two vectors $U_{\alpha_1}$, $U_{\alpha_2} \in T_\rho \Hit_n(S)$  associated to the infinitesimal shearing of $\rho \in \Hit_n(S)$ according to twisted transverse cocycles $\alpha_1$, $\alpha_2 \in \twis$ for the geodesic lamination $\lambda$.

Let $\Phi$ be a  train track carrying the maximal geodesic lamination $\lambda$, and let $\widehat\Phi$ be its orientation cover as in \S \ref{sect: twisted cocycles}. 

At each switch $s$ of $\widehat \Phi$, there is a single incoming edge $e^{\mathrm{in}}_s$ on one side of $s$ and two outgoing edges $e^{\mathrm{left}}_s$ and $e^{\mathrm{right}}_s$ on the other side, with  $e^{\mathrm{left}}_s$ and $e^{\mathrm{right}}_s$ respectively diverging to the left and to the right as seen from the incoming edge $e^{\mathrm{in}}_s$ and for the orientation of S. Let $k^{\mathrm{left}}_s$ and $k^{\mathrm{right}}_s$ be the respective intersections of the tie $s$ with $e^{\mathrm{left}}_s$ and $e^{\mathrm{right}}_s$. Lift $s$ to an arc $\widetilde s$ in the universal cover $\widetilde S$, and let $\widetilde k^{\mathrm{left}}_s$ and $\widetilde k^{\mathrm{right}}_s$ be the corresponding lifts of $k^{\mathrm{left}}_s$ and $k^{\mathrm{right}}_s$. 

By construction, the two arcs $\widetilde k^{\mathrm{left}}_s$ and $\widetilde k^{\mathrm{right}}_s$ are transverse to the preimage $\widetilde \lambda$ of $\widetilde S$ and oriented by the canonical orientation of the ties of $\widehat \Phi$. The cocycles $u_{\alpha_1}$ and $u_{\alpha_2}$ defined by Lemma~\ref{gap formula} then provide $u_{\alpha_1}(\widetilde k^{\mathrm{right}}_s)$ and $u_{\alpha_2}(\widetilde k^{\mathrm{right}}_s) \in \mathfrak{sl}_n(\mathbb{R})$. These elements of $\mathfrak{sl}_n(\mathbb{R})$ depend on the lift $\widetilde s$ of the tie $s$ but their pairing $B \big( u_{\alpha_1}(\widetilde k^{\mathrm{right}}_s), u_{\alpha_2}(\widetilde k^{\mathrm{right}}_s) \big) $ does not, by invariance of the Killing form under the adjoint representation. We will consequently write 
$$
B \big( u_{\alpha_1}( e^{\mathrm{right}}_s), u_{\alpha_2}( e^{\mathrm{right}}_s) \big)
=
B \big( u_{\alpha_1}(\widetilde k^{\mathrm{right}}_s), u_{\alpha_2}(\widetilde k^{\mathrm{right}}_s) \big) \in \R. 
$$

We borrow our first step from \cite{SozBon}. 

\begin{prop}
\label{ABGform}
If $\alpha_1$, $\alpha_2 \in \twis$, the Atiyah-Bott-Goldman pairing $\omega(U_{\alpha_1}, U_{\alpha_2})$ of the corresponding infinitesimal shearing vectors $U_{\alpha_1}$, $U_{\alpha_2} \in T_\rho \Hit_n(S)$ is equal to 
\[
{\frac{1}{2}}\sum_s
\Big(
B \big( u_{\alpha_1}( e^{\mathrm{right}}_s), u_{\alpha_2}( e^{\mathrm{left}}_s) \big)
-B \big( u_{\alpha_1}( e^{\mathrm{left}}_s), u_{\alpha_2}( e^{\mathrm{right}}_s) \big)
\Big),
\]
where in the sum $s$ ranges over all switches of the orientation cover $\widehat{\Phi}$ of the train track $\Phi$.
\end{prop}  
\begin{proof}
The proof follows from an immediate adaptation to the current framework of the computation of Lemma 3 in \cite{SozBon}, which uses an explicit triangulation of the surface $S$ to determine the cup-product of the cohomology classes $[u_{\alpha_1}]$, $[u_{\alpha_2}] \in  H^1(S;\mathfrak{sl}_n(\mathbb{R})_{\mathrm{Ad}_{\rho}})$.
\end{proof}

\subsection{Topological, geometric and combinatorial estimates}
Theorems~\ref{thm:MainThm1} and  \ref{thm:MainThm2} are obtained from Proposition~\ref{ABGform} by a limiting process that is very similar to the one used in \cite{SozBon}. It uses relatively classical geometric estimates that we indicate here. 

Let $\Phi$ be a train track carrying the geodesic lamination $\lambda$,  let $k$ be a generic tie of $\Phi$, and let $d$ be a component of $k-\lambda$. In the universal cover $\widetilde{S}$ of $S$, let $\widetilde{\Phi}$, $\widetilde{\lambda}$ be the preimages of $\Phi$ and $\lambda$, and let $\widetilde{d}$ be a lift of $d$. If $d$ does not contain any of the endpoints of $k$, let $g_{\widetilde d}^+$ and $g_{\widetilde d}^-$ be the two leaves of $\widetilde{\lambda}$ passing through the endpoints of $\widetilde d$. The \emph{divergence radius} or \emph{depth} $r(d)$  of $d$ is the largest integer $r\geq 1$ such that $g_{\widetilde d}^+$ and $g_{\widetilde d}^-$ successively cross the same sequence of edges $e_{-r+1}$, $e_{-r+2}$, \dots, $e_0$, \dots, $e_{r-2}$, $e_{r-1}$ of $\widetilde{\Phi}$ where $e_0$ is the edge of $\widetilde{\Phi}$ containing $\widetilde{d}$.   Note that $r(d)$ is independent of the choice of the lift $\widetilde{d}$ of $d$. By convention, $r(d)=1$ if $d$ contains one of the endpoints of the tie $k$.

We now  state three lemmas whose proof can be found in \cite{Bon1}. As usual, we fix an arbitrary negatively curved metric $m$ on the closed surface S. Let $k$ be a generic tie of the train track $\Phi$ carrying the geodesic lamination $\lambda$.    

\begin{lem}[{\cite[Lemma 4]{Bon1}}]\label{boundednumberofgaps}
There is a universal constant $C$, depending only on the topology of $S$, such that for every $r\geq 0$  the number of components $d$ of $k-\lambda$ with $r(d)=r$ is uniformly bounded by $C$.  \qed
\end{lem}

\begin{lem}\cite[Lemma 5]{Bon1}\label{expdecaygaplength}
There exist constants $A$, $B > 0$ such that
\[\ell_m(d)\leq Be^{-Ar(d)},\]
for every component $d$ of $k-\lambda$ with depth $r(d)$ and length $\ell_m(d)$. \qed
\end{lem}


We fix an arbitrary norm $\Vert \ \Vert$ on $\R^{n-1}$. Proposition~\ref{prop:transversecocyclesweightinR(n-1)} associates to each twisted transverse cocycle $\alpha \in \twis$ an edge weight system $a_\alpha \in \mathcal W(\widehat\Phi; \R^{n-1})$ for the orientation cover of the train track $\Phi$. Set 
$$
\Vert \alpha \Vert_{\widehat{\Phi}}=\max_e \Vert a_\alpha(e) \Vert
$$
where the maximum is taken over the edges $e$  of $\widehat{\Phi}$.

\begin{lem}\cite[Lemma 6]{Bon1}\label{lineargrowthgapheight}
Choose an orientation for a generic tie $k$ of the train track $\Phi$. For every component $d$ of $k-\lambda$, let $k_d$ be the oriented subarc of $k$ that joins the negative end point of $k$ to an arbitrary point of $d$. Then
$$\Vert \alpha (k_d) \Vert \leq r(d) \Vert \alpha \Vert_{\widehat{\Phi}}$$
for every component $d$ of $k-\lambda$ and for every twisted transverse cocycle $\alpha \in \twis$. \qed
\end{lem}

\subsection{Proof of Theorems~\ref{thm:MainThm1} and \ref{thm:MainThm2}} 
\label{sect:Proof main thms}

 We are now ready to prove Theorem~\ref{thm:MainThm1}. For convenience, we repeat this statement as Theorem~\ref{maintheorem1} below.

By Proposition~\ref{prop:transversecocyclesweightinR(n-1)} and Lemma~\ref{lem:TransverseCocycleDefinesHomology}, a twisted transverse cocycle $\alpha \in \twis$ defines a homology class $[\alpha] \in H_1(\widehat \Phi; \R^{n-1})= H_1(\widehat \Phi; \R)^{n-1}$. For $a=1$, $2$, \dots, $n-1$, let $[\alpha^{(a)}] \in H_1(\widehat \Phi; \R)$ denote the $a$-th component of $[\alpha]$.

\begin{thm}
 \label{maintheorem1}
 Let S be a closed oriented surface with genus $g\geq 2$,  let $\lambda$ be a maximal geodesic lamination on $S$, and let $\Phi$ be a train track carrying $\lambda$. Then, for every two twisted transverse cocycles
  $\alpha_1$, $\alpha_2 \in \twis$, the Atiyah-Bott-Goldman pairing $\omega(U_{\alpha_1}, U_{\alpha_2})$ of the corresponding infinitesimal shearing vectors $U_{\alpha_1}$, $U_{\alpha_2} \in T_\rho \Hit_n(S)$ is equal to 
 $$
 \omega(U_{\alpha_1}, U_{\alpha_2})  = \sum_{a,b=1}^{n-1} C(a,b)\, [\alpha_1^{(a)}] \cdot  [\alpha_2^{(b)}],
 $$
where the homology classes $[\alpha_1^{(a)}]$, $[\alpha_2^{(b)}] \in H_1(\widehat \Phi; \R)$ are defined as above, where $[\alpha_1^{(a)}] \cdot  [\alpha_2^{(b)}]$ denotes their algebraic intersection number in the oriented surface $\widehat\Phi$, and where
$$
C(a,b)=
\begin{cases}
  2a(n-b) &\text{ if } \, a\leq b\\
  2b(n-a) &\text{ if } \, a\geq b .
\end{cases}
$$
\end{thm}

\begin{proof}
We will use the overall strategy of \cite{SozBon}, by `unzipping zippers' (see \cite[section 1.7,2.4]{PeH}) to construct a family of nested train tracks $\Phi_R$ that carry the lamination $\lambda$ and are thinner and thinner. We need to do this in a controlled way in order to estimate the growth of the weights defined on the edges of $\widehat \Phi_R$ by $\alpha_1$ and $\alpha_2$. 

We begin with the train track  $\Phi= \Phi_0$. 

 For each integer $R\geq 1$, delete from each edge $e$ of $\Phi_0$ the components of $e-\lambda$ with depth less than or equal to $R$ (defining the depth of a component of $e-\lambda$ as the depth of the corresponding component of $k_e-\lambda$ where $k_e$ is an arbitrary tie of $e$). The union of the remaining pieces of $e-\lambda$, for all edges of $\Phi_0$, can then be slightly enlarged to a train track $\Phi_R \subset \Phi_0$ carrying $\lambda$. 
 
 To simplify the notation, write 
 $$
 \alpha_1 * \alpha_2 =  \sum_{a,b=1}^{n-1} C(a,b)\, [\alpha_1^{(a)}] \cdot  [\alpha_2^{(b)}]
 $$
for the quantify that we want to show is equal to $ \omega(U_{\alpha_1}, U_{\alpha_2})  $. Lemma~\ref{lem:ComputeIntersectionEdgeWeightSystems} expresses this as a sum
$$
 \alpha_1 * \alpha_2 = \frac12 \sum_s \sum_{a,b=1}^{n-1}  C(a,b) \big( \alpha_1^{(a)}(e_s^{\mathrm{right}})\alpha_2^{(b)}(e_s^{\mathrm{left}}) - \alpha_1^{(a)}(e_s^{\mathrm{left}})\alpha_2^{(b)}(e_s^{\mathrm{right}}) \big)
$$
where the sum is over all switches $s$ of $\widehat\Phi_R$, and $e_s^{\mathrm{left}}$ and $e_s^{\mathrm{right}}$ are the edges of $\widehat\Phi_R$ that are adjacent to the same side of $s$ and respectively diverge to the left and the right as seen from $s$. 

Proposition~\ref{ABGform} provides another formula
\[
 \omega(U_{\alpha_1}, U_{\alpha_2})  = 
{\frac{1}{2}}\sum_s
\Big(
B \big( u_{\alpha_1}( e^{\mathrm{right}}_s), u_{\alpha_2}( e^{\mathrm{left}}_s) \big)
-B \big( u_{\alpha_1}( e^{\mathrm{left}}_s), u_{\alpha_2}( e^{\mathrm{right}}_s) \big)
\Big)
\]
expressing $ \omega(U_{\alpha_1}, U_{\alpha_2})  $ as a sum over the switches of $\widehat\Phi_R$.

Our strategy will be to  show that, for $R$ sufficiently large,  the contribution $\sum_{a,b=1}^{n-1}  C(a,b) \,\alpha_1^{(a)}(e_s^{\mathrm{right}})\,\alpha_2^{(b)}(e_s^{\mathrm{left}}) $ of each switch of $\widehat\Phi_R$  to the first sum is arbitrarily close to its contribution $B \big( u_{\alpha_1}( e^{\mathrm{right}}_s), u_{\alpha_2}( e^{\mathrm{left}}_s) \big)$ to the second sum.  

We will use the standard terminology that  $X$ is an $O(Y)$ if there exist a constant $C$ such that $|X|\leq C|Y|$.  The constants will depend on the initial train track $\Phi=\Phi_0$ and on the negatively curved metric $m$ chosen on $S$ but not on $R$. 

Recall that we arbitrarily chose a norm $\Vert \ \Vert$ on $\R^{n-1}$ and that for a twisted transverse cocycle $\alpha \in \twis$, we introduced
$$
\Vert \alpha \Vert_{\widehat{\Phi}}=\max_e \Vert \alpha(e) \Vert
$$
where the maximum is taken over the edges $e$  of $\widehat{\Phi}$. 
 
\begin{lem}
\label{weightoftransversecocycle}
For every edge $e$ of $\widehat\Phi_R$ and for every twisted transverse cocycle $\alpha \in \twis$, the edge weight $\alpha(e)\in \R^{n-1}$ is an $O(\Vert\alpha\Vert_{\widehat{\Phi}} R)$.
\end{lem}
\begin{proof} By construction, $\Phi_R$ is contained in $\Phi_0=\Phi$, and we can arrange that its orientation cover $\widehat \Phi_R$ is just the preimage of $\Phi_R$ in $\widehat\Phi$. 

If $k$ is a generic tie of $\widehat\Phi$ meeting $e$, then $k \cap e$ is a union of ties of $e$. Let $k_e$ be one of these components of $k\cap e$,  and let $d_e^+$ and $d_e^-$ be the components of $k-\lambda$ that contain the positive and negative end points of $k_e$. Then,
\[\alpha(e)=\alpha(k_e)=\alpha(k_{d_e^+})-\alpha(k_{d_e^-})\]  
by finite additivity of $\alpha$.

By construction of $\Phi_R$, the depths
$r(d_e^+)$, $r(d_e^-)$ are both bounded by $R$.  
The result then follows from Lemma  \ref{lineargrowthgapheight}.
\end{proof}

To control estimates in the universal cover $\widetilde S$, we choose a compact subset $K\subset \widetilde S$ whose projection $K \to S$ is surjective. For instance, one could take for $K$ a closed ball whose radius is larger than the diameter of $S$. 

For each geodesic $g$ of $\widetilde{S}$ and $a=1$, $2$, \dots, $n-1$,  let $t_g^{(a)} \in \mathfrak{sl}_n(\mathbb{R})$ denote the infinitesimal  $a$-shearing map along $g$, as defined in \S \ref{sect: elementary shearing}. 

\begin{lem}
\label{difference between infinitesimal shearing maps}
In the preimage $\widetilde \Phi_R\subset \widetilde S$ of the train track $\Phi_R$, let $e$ be an edge that meets  the above compact subset $K \subset \widetilde S$. 
For any two leaves $g_1$, $g_2$ of $\widetilde{\lambda}$ that cross ${e}$  and for every $a=1$, $2$, \dots, $n-1$, the difference $t_{g_1}^{(a)} -t_{g_2}^{(a)}$ is an $O(e^{-AR})$ in $\mathfrak{sl}_n(\R)$ for some constant $A>0.$ 
\end{lem}

\begin{proof} 
 Orient $g_1$ and $g_2$ so that they cross the ties of $ e$ in the same direction. Let $x_1^+$,  $x_2^+$, $x_1^-$,  $x_2^- \in \partial_\infty \widetilde S$ be their positive and negative endpoints, respectively. 

By construction of the train track $\Phi_R$, the leaves $g_1$ and $g_2$ successively cross the same sequence of edges $e_{-R}$, $e_{-R+1}$, \dots, $e_0$, \dots, $e_{R-1}$, $e_{R-1}$ of $\widetilde{\Phi}_0$ where $e_0$ meets our chosen compact subset $K \subset \widetilde S$. Indeed, $g_1$ and $g_2$ would otherwise be separated by a component of $e_0 -\widetilde\lambda$ whose depth is strictly less than $R$, contradicting the definition of $\Phi_R$. It follows that the distances $d(x_1^+, x_2^+)$  and $d(x_1^-, x_2^-)$ are both  $O(e^{-AR})$ for some constant $A>0$ depending on the (negative) curvature of the metric of $S$. (We are here using the fact that the geodesics pass at uniformly bounded distance from the compact subset $K$.)

The infinitesimal  $a$th shearing map $t_g^{(a)}$ depends differentiably on the flags $\mathcal F_\rho(x^+)$ and $\mathcal F_\rho(x^-)\in \mathrm{Flag}(\mathbb{R}^n)$ associated to the endpoints of $g$ by the flag curve  $\mathcal{F}_{\rho}\colon \partial_{\infty}\widetilde{S}\rightarrow \mathrm{Flag}(\mathbb{R}^n)$, and $\mathcal F_\rho$ is H\"older continuous with H\"older exponent $\nu>0$ by 
Proposition~\ref{prop:flag curve}. It follows that 
\begin{align*}
t_{g_1}^{(a)} -t_{g_2}^{(a)} 
&= O\Big( d \big (\mathcal F_\rho(x_1^+), \mathcal F_\rho(x_2^+)  \big) + 
d \big (\mathcal F_\rho(x_1^-), \mathcal F_\rho(x_2^-)  \big) \Big)\\
&=  O\big( d  (x_1^+, x_2^+  )^\nu + 
d (x_1^-, x_2^-)^\nu  \big) = O(e^{-\nu AR}).
\qedhere
\end{align*}
\end{proof}

\begin{lem}
\label{gapapp}
Let $k$ be an oriented tie of $\widetilde \Phi_R$ that meets the chosen compact subset $K \subset \widetilde S$. Then, for every transverse cocycle $\alpha \in \twis$,
\[
u_\alpha({k})=\sum_{a=1}^{n-1}\alpha^{(a)}(k)t_g^{(a)}+O\big(\Vert\alpha\Vert_{\widehat{\Phi}}e^{-AR} \big)
\] 
for some constant $A>0$, where $g$ is an arbitrary leaf of $\widetilde{\lambda}$ that crosses $k$ and is oriented to the left of ${k}$. 
\end{lem}

\begin{proof}
By the Gap Formula of Lemma~\ref{gap formula},
 \begin{align*}
u_\alpha({k}) - \sum_{a=1}^{n-1}\alpha^{(a)}(k)t_g^{(a)} &=
\sum_{a=1}^{n-1} \alpha^{(a)}({k}) (t_{g_{d_+}^{-}}^{(a)} -t_g^{(a)}) \\
&\qquad\qquad + \sum_{a=1}^{n-1} \sum_{d\neq d_+,d_-} \alpha^{(a)} ( {k_d}) \big( t_{g_d^{-}}^{(a)}-t_{g_d^{+}}^{(a)} \big), 
\end{align*}
with the notation of that lemma. 

By construction, the tie $k$ of $\widetilde \Phi_R$ is contained in the lift $k'$ of a tie of the initial train track $\Phi_0$. In particular, a component $d$ of $k -\widetilde\lambda$ that does not contain one of the endpoints of $k$ is also a component of  $k' -\widetilde\lambda$, and we can consider its depth $r(d)$ with respect to the train track $\Phi_0$. Note that $r(d) \geq R$ by construction of $\Phi_R$. 

For such a component $d$ of $k -\widetilde\lambda$, let $k'_d$ be a subarc of $k'$ going from the negative endpoint of $k'$ to an arbitrary point of $d$, and let $k''$ be the subarc of $k'$ going from the negative endpoint of $k'$ to the negative endpoint of $k$. Then, by Lemma~\ref{lineargrowthgapheight}, we have
\begin{eqnarray}\label{alpha(k-d)}
 \alpha^{(a)}(k_d) &=&  \alpha^{(a)}(k'_d) -  \alpha^{(a)}(k'') 
 = O\big(r(d) \Vert\alpha\Vert_{\widehat{\Phi}} \big) + O\big(R\Vert\alpha\Vert_{\widehat{\Phi}}\big)\nonumber \\
& =&O\big(r(d) \Vert\alpha\Vert_{\widehat{\Phi}}  \big).
\end{eqnarray}

From Lemma~\ref{expdecaygaplength} it follows that $\ell_m(d) = O(e^{-Ar(d)})$ for some constant $A>0$. A classical result in negatively curved geometry (see for instance \cite[Lemma~5.2.6]{CEG}) shows that because the geodesics  $g_d^+$ and  $g_d^-$ are disjoint, the distance between their endpoints in an $O(\ell_m(d))$. As in the proof of Lemma~\ref{difference between infinitesimal shearing maps}, it follows that 
\begin{equation}\label{differenceofinfiniesimaltr}
 t_{g_d^{-}}^{(a)}-t_{g_d^{+}}^{(a)} = O(e^{-\nu Ar(d)}),
\end{equation}
where $\nu$ is the H\"older exponent of the flag curve $\mathcal F_\rho$. 

Combining the equations (\ref{alpha(k-d)}) and (\ref{differenceofinfiniesimaltr}), we get
$$
 \sum_{d\neq d_+,d_-} \alpha^{(a)} ( {k_d}) \big( t_{g_d^{-}}^{(a)}-t_{g_d^{+}}^{(a)} \big)
 = O\bigg(   \sum_{d\neq d_+,d_-}  r(d) \Vert\alpha\Vert_{\widehat{\Phi}}  e^{-\nu Ar(d)} \bigg).
$$

For each component $d$ of $k-\widetilde\lambda$, we saw that $r(d) \geq R$ by construction of the train track $\Phi_R$ and that for every $r\geq R$ the number of components $d$ with $r(d)=r$ is uniformly bounded. From this it follows that 
\begin{align*}
 \sum_{d\neq d_+,d_-} \alpha^{(a)} ( {k_d}) \big( t_{g_d^{-}}^{(a)}-t_{g_d^{+}}^{(a)} \big)
 &= O\bigg(   \sum_{r=R}^\infty  r \Vert\alpha\Vert_{\widehat{\Phi}}  e^{-\nu Ar} \bigg)\\
 &= O\big(   \Vert\alpha\Vert_{\widehat{\Phi}}  e^{- A'r} \big)
\end{align*}
for any $A'< \nu A$. 

Similarly,
$$
\alpha^{(a)}({k}) (t_{g_{d_+}^{-}}^{(a)} -t_g^{(a)}) =  O\big( R  \Vert\alpha\Vert_{\widehat{\Phi}}  e^{- \nu AR} \big) =  O\big(   \Vert\alpha\Vert_{\widehat{\Phi}}  e^{- A'r} \big).
$$

As a consequence,
\begin{align*}
u_\alpha({k}) - \sum_{a=1}^{n-1}\alpha^{(a)}(k)t_g^{(a)} &=
\sum_{a=1}^{n-1} \alpha^{(a)}({k}) (t_{g_{d_+}^{-}}^{(a)} -t_g^{(a)}) \\
&\qquad\qquad + \sum_{a=1}^{n-1} \sum_{d\neq d_+,d_-} \alpha^{(a)} ( {k_d}) \big( t_{g_d^{-}}^{(a)}-t_{g_d^{+}}^{(a)} \big)\\
&= O\big(   \Vert\alpha\Vert_{\widehat{\Phi}}  e^{- A'r} \big),
\end{align*}
which proves the lemma. 
\end{proof}

The following computation explains the origin of the constants
$$
C(a,b)=
\begin{cases}
  2a(n-b) &\text{ if } \, a\leq b\\
  2b(n-a) &\text{ if } \, a\geq b 
\end{cases}
$$
that occurred in the statements of Theorems~\ref{thm:MainThm1} and \ref{thm:MainThm2}. 

\begin{lem}
\label{KillingForm}
Let $e$ be an edge of $\widetilde\Phi_R$ that meets the compact subset $K \subset \widetilde S$. 
For any two leaves $g_1$, $g_2$ of $\widetilde{\lambda}$ that cross $e$, and for every $a$, $b=1$, $2$, \dots, $n-1$, 
$$
B(t_{g_1}^{(a)},t_{g_2}^{(b)})= C(a,b) + O(e^{-AR})
$$
for some constant $A>0$.
\end{lem}
\begin{proof} Let $g$ be a geodesic of $\widetilde S$. 
The Cartan-Killing form of $\mathfrak{sl}_n(\R)$ is such that
$$B(t_g^{(a)},t_g^{(b)})=2n \mathrm{Tr}(t_g^{(a)} t_g^{(b)})$$
where $\mathrm{Tr}$ denotes the trace. Also, $t_g^{(a)}$ was defined so that, in an appropriate basis, it is represented by a diagonal matrix whose first $a$ diagonal entries are $\frac{n-a}n$, while the remaining diagonal entries are $-\frac an$

When $a\leq b$, an immediate computation then gives 
\begin{align*}
 \mathrm{Tr}({t_g}^{(a)}{t_g}^{(b)})&=  a \frac{(n-a)}n \frac{(n-b)}n - (b-a)  \frac an \frac{(n-b)}n  +(n-b) \frac an \frac bn\\
 &= \frac{a(n-b)}{n} 
\end{align*}
so that $B(t_g^{(a)},t_g^{(b)})= C(a,b) $ in this case. 

A similar computation yields $B(t_g^{(a)},t_g^{(b)})= C(a,b) $  when $a\geq b$ as well.

Then, by Lemma~\ref{difference between infinitesimal shearing maps},
\begin{align*}
B(t_{g_1}^{(a)},t_{g_2}^{(b)}) &= B(t_{g_1}^{(a)},t_{g_1}^{(b)})  + O(e^{-AR})\\
&=  C(a,b) + O(e^{-AR})
\qedhere
\end{align*}
\end{proof}

We are now ready to prove  Theorem \ref{maintheorem1}. Recall that we want to prove that 
\[
 \omega(U_{\alpha_1}, U_{\alpha_2})  = 
{\frac{1}{2}}\sum_s
\Big(
B \big( u_{\alpha_1}( e^{\mathrm{right}}_s), u_{\alpha_2}( e^{\mathrm{left}}_s) \big)
-B \big( u_{\alpha_1}( e^{\mathrm{left}}_s), u_{\alpha_2}( e^{\mathrm{right}}_s) \big)
\Big)
\]
 is equal to 
\begin{align*}
 \alpha_1 * \alpha_2  &=  \sum_{a,b=1}^{n-1} C(a,b)\, [\alpha_1^{(a)}] \cdot  [\alpha_2^{(b)}]\\
 &= \frac12 \sum_s \sum_{a,b=1}^{n-1}  C(a,b) \big( \alpha_1^{(a)}(e_s^{\mathrm{right}})\alpha_2^{(b)}(e_s^{\mathrm{left}}) - \alpha_1^{(a)}(e_s^{\mathrm{left}})\alpha_2^{(b)}(e_s^{\mathrm{right}}) \big),
\end{align*}
 where in both cases $s$ ranges over all switches of the orientation cover $\widehat \Phi_R$ of the train track $\Phi_R$. 
 
 Let us focus on the contribution of a switch $s$. Recall from Lemma~\ref{ABGform} that $B \big( u_{\alpha_1}( e^{\mathrm{right}}_s), u_{\alpha_2}( e^{\mathrm{left}}_s) \big)$ is defined by lifting $s$ to a switch $\widetilde s$ of the preimage $\widetilde\Phi_R$ of $\Phi_R$ in the universal cover $\widetilde S$. In order to apply the estimates that we just obtained we choose this lift $\widetilde s$ so its ``midpoint'', corresponding to the point $s \cap e^{\mathrm{left}}_s \cap e^{\mathrm{left}}_s$, is contained in the chosen compact subset $K \subset \widetilde S$. 
 
 The three edges $e_s^{\mathrm{in}}$, $e_s^{\mathrm{right}}$ and $e_s^{\mathrm{left}}$ of $\Phi_R$ that meet at $s$ lift to the three edges $\widetilde{e}_s^{\ \mathrm{in}}$, $\widetilde{e}_s^{\ \mathrm{right}}$ and $\widetilde{e}_s^{\ \mathrm{left}}$ of $\widetilde{\Phi}_R$ that meet at $\widetilde{s}$.
Let $g^{\mathrm{right}}$ and $g^{\mathrm{left}}$ be two leaves of $\widetilde{\lambda}$ that respectively cross the edges $\widetilde{e}_s^{\ \mathrm{right}}$ and $\widetilde{e}_s^{\ \mathrm{left}}$. The switch $s$ is canonically oriented since it is a tie of the orientation cover $\widehat\Phi_R$. Lift this orientation to an orientation of $\widetilde{s}$, and orient $g^{\mathrm{right}}$ and $g^{\mathrm{left}}$ to the left of $\widetilde{s}$.

As in the set up of Lemma~\ref{ABGform}, let $\widetilde{k}_s^{\mathrm{right}} = \widetilde s \cap \widetilde{e}_s^{\ \mathrm{right}}$ and $\widetilde{k}_s^{\mathrm{left}} = \widetilde s \cap \widetilde{e}_s^{\ \mathrm{left}}$. Note that both arcs meet the compact subset $K \subset \widetilde S$. 
By Lemma \ref{gapapp},
\begin{align*}
 u_{\alpha_1}(\widetilde k^{\mathrm{right}}_s)
 &=\sum_{a=1}^{n-1}\alpha_1^{(a)}(e^{\mathrm{right}}_s)t_{g^{\mathrm{right}}}^{(a)}+O(e^{-AR}) \\
u_{\alpha_2}(\widetilde k^{\mathrm{left}}_s)
&=\sum_{b=1}^{n-1}\alpha_2^{(b)}(e^{\mathrm{left}}_s)t_{g^{\mathrm{left}}}^{(b)}+O(e^{-AR})
\end{align*}
where we have incorporated the norms $\Vert\alpha_1\Vert_{\widehat{\Phi}}$ and $\Vert\alpha_2\Vert_{\widehat{\Phi}}$ in the constants of the symbols $O$. 

It follows that 
\begin{align*}
 B \big( u_{\alpha_1}( e^{\mathrm{right}}_s), u_{\alpha_2}( e^{\mathrm{left}}_s) \big)
 &= B \big( u_{\alpha_1}( \widetilde k^{\mathrm{right}}_s), u_{\alpha_2}( \widetilde k^{\mathrm{left}}_s) \big)\\
 &= \sum_{a,b=1}^{n-1} \alpha_1^{(a)}(e^{\mathrm{right}}_s) \alpha_2^{(b)}(e^{\mathrm{left}}_s) 
 B(t_{g^{\mathrm{right}}}^{(a)}, t_{g^{\mathrm{left}}}^{(b)})\\
 &\qquad\qquad\qquad\qquad\qquad\qquad + O(R e^{-AR})\\
 &= \sum_{a,b=1}^{n-1} \alpha_1^{(a)}(e^{\mathrm{right}}_s) \alpha_2^{(b)}(e^{\mathrm{left}}_s) 
 C(a,b)\\
 &\qquad\qquad\qquad\qquad\qquad\qquad + O(R^2 e^{-AR})
\end{align*}
by noting that the geodesics $g^{\mathrm{right}}$ and $g^{\mathrm{left}}$ both cross the edge $\widetilde{e}_s^{\ \mathrm{in}}$, applying Lemma~\ref{KillingForm}, and remembering that $\alpha_1^{(a)}(e^{\mathrm{right}}_s)$ and $ \alpha_2^{(b)}(e^{\mathrm{left}}_s)$ are both $O(R)$ by Lemma~\ref{weightoftransversecocycle}.

Similarly,
\begin{align*}
 B \big( u_{\alpha_1}( e^{\mathrm{left}}_s), u_{\alpha_2}( e^{\mathrm{right}}_s) \big)
 &= \sum_{a,b=1}^{n-1} \alpha_1^{(a)}(e^{\mathrm{left}}_s) \alpha_2^{(b)}(e^{\mathrm{right}}_s) 
 C(a,b)\\
 &\qquad\qquad\qquad\qquad\qquad\qquad + O(R^2 e^{-AR}).
\end{align*}

Comparing the expressions of $B \big( u_{\alpha_1}( e^{\mathrm{left}}_s), u_{\alpha_2}( e^{\mathrm{right}}_s) \big)$ and $\alpha_1*\alpha_2$, it follows that
$$
 B \big( u_{\alpha_1}( e^{\mathrm{left}}_s), u_{\alpha_2}( e^{\mathrm{right}}_s) \big)
 = \alpha_1*\alpha_2 + O(R^2 e^{-AR}),
$$
since the number of switches of $\Phi_R$ is constant and equal to $6|\chi(S)|$. 

The quantities $B \big( u_{\alpha_1}( e^{\mathrm{left}}_s), u_{\alpha_2}( e^{\mathrm{right}}_s) \big)$ and $\alpha_1*\alpha_2$ are independent of the number $R$. Letting $R$ tend to $\infty$, we conclude that 
\begin{align*}
 B \big( u_{\alpha_1}( e^{\mathrm{left}}_s), u_{\alpha_2}( e^{\mathrm{right}}_s) \big)
 &= \alpha_1*\alpha_2 \\
 &=  \sum_{a,b=1}^{n-1} C(a,b)\, [\alpha_1^{(a)}] \cdot  [\alpha_2^{(b)}]
\end{align*}
This concludes the proof of Theorem~\ref{maintheorem1}, and therefore of Theorem~\ref{thm:MainThm1}. 
\end{proof}

This also proves Theorem~\ref{thm:MainThm2}, since Lemma~\ref{lem:ComputeIntersectionEdgeWeightSystems} shows the equivalence of Theorems~\ref{thm:MainThm1} and \ref{thm:MainThm2}.


\begin{thebibliography}{CEG} 

\bibitem[AB]{AB} M. Atiyah, R. Bott, 
\emph{The Yang-Mills equations over Riemann surfaces}, 
Philos. Trans. Roy. Soc. London A \textbf{308} (1982), 523--615.

\bibitem[CEG]{CEG} R.D. Canary, D.B.A. Epstein, P. Green, 
\emph{Notes on notes of Thurston}, 
Analytical and Geometric Aspects of Hyperbolic Space, L.M.S. Lecture Notes Series, Cambridge Univ. Press, \textbf{111} (1987), 3--92.


\bibitem[Bo1]{Bon1} F. Bonahon, 
\emph{Shearing hyperbolic surfaces, bending pleated surfaces, and Thurston's symplectic form}, 
Ann. Fac. Sci. Toulouse \textbf{5} (1996), 233--297.

\bibitem[Bo2]{Bon2} F. Bonahon, 
\emph{Transverse H\"{o}lder distributions for geodesic laminations},
 Topology \textbf{36} (1997), 103--122.

\bibitem[BD1]{BonDr1} F. Bonahon, G. Dreyer, 
\emph{Parameterizing Hitchin components},
Duke Math. J. \textbf{163 (15)} (2014), 2935--2975. 

\bibitem[BD2]{BonDr2} F. Bonahon, G. Dreyer, 
\emph{Hitchin characters and geodesic laminations},
Acta Math. \textbf{218 (2)} (2017), 201--295.  

\bibitem[Dr]{Dr} G. Dreyer, 
\emph{Thurston's cataclysm deformations for Anosov representations}, Preprint, (2013), \texttt{arXiv:1301.6961}.

\bibitem[FG]{FoG} V. Fock, A. Goncharov, 
\emph{Moduli spaces of local systems and higher Teichm\"{u}ller theory}, Publ. Math. Inst. Hautes \'{E}tudes Sci. \textbf{103} (2006), 1--211.

\bibitem[Go]{Go} W.M. Goldman,
\emph{The symplectic nature of fundamental groups of surfaces},
Adv. in Math. \textbf{54} (1984), 200--225.  

\bibitem[Hi]{Hi} N.J. Hitchin, 
\emph{Lie groups and Teichm\"{u}ller space},
Topology \textbf{31} (1992), 449--473.  

\bibitem[La]{Lab} F. Labourie, 
\emph{Anosov flows, surface groups and curves in projective space},
Invent. Math. \textbf{165} (2006), 51--114. 

\bibitem[PH]{PeH} R.C Penner, J.L. Harer, 
\emph{Combinatorics of train tracks},
Ann. Math.Studies, Princeton Univ. Press \textbf{125} (1992).

\bibitem[SB]{SozBon} Y. S\"{o}zen, F. Bonahon, \emph{Weil-Petersson and Thurston symplectic forms}, 
Duke Math. J. \textbf{108 (3)} (2001), 581--597.

\bibitem[Th1]{Th1} W.P. Thurston,
\emph{The geometry and topology of 3-manifolds},  Princeton lecture notes (1978-1981), available at http://library.msri.org/books/gt3m/.

\bibitem[Th2]{Th2} W.P. Thurston, \emph{Earthquakes in two-dimensional hyperbolic geometry}, in: \emph{Low-dimensional
topology and Kleinian groups, Warwick and Durham, 1984}, (D.B.A. Epstein ed.), L.M.S.
Lecture Note Series 112, Cambridge University Press, Cambridge, (1986), 91--112.

\bibitem[Th3]{Th3} W.P. Thurston,
\emph{Minimal stretch maps between hyperbolic surfaces}, 
unpublished preprint (1986).



\bibitem[We]{We2} A. Weil, 
\emph{Remarks on the cohomology of groups}, Ann. of Math. \textbf{80} (1964), 149--157.

\end{thebibliography}
\end{document}